\documentclass[letterpaper]{amsart}
\usepackage[OT1]{fontenc}

\usepackage{amsaddr}

\usepackage{amsfonts}
\usepackage{amssymb}
\usepackage{tikz}
\usepackage{enumerate}
\usepackage{xcolor}
\usepackage{graphicx}
\usepackage{subcaption}
\usepackage[hidelinks]{hyperref}

\usepackage{concmath}
\usepackage[slantedGreek]{ccfonts}
\usepackage{eulervm}

\usepackage{soul}

\theoremstyle{plain}
\newtheorem{thm}{Theorem}
\newtheorem{prop}{Proposition}
\newtheorem{lem}{Lemma}

\newtheorem{cor}{Corollary}

\theoremstyle{definition}
\newtheorem{defn}{Definition}
\newtheorem{ex}{Example}

\theoremstyle{remark}

\newtheorem{case}{Case}[subsection]

\DeclareMathOperator{\sat}{sat}

\DeclareMathOperator{\wtcp}{wt_{CP}}
\DeclareMathOperator{\wt}{wt}
\DeclareMathOperator{\wtzero}{wt_{0}}
\DeclareMathOperator{\wtone}{wt_{1}}

\newcommand{\xl}{\ensuremath{X \hspace{-.7pt} L}}

\labelformat{subfigure}{\thefigure(\textsc{#1})}

\title[A lower bound on the saturation number]{A lower bound on the saturation number and a strengthening for triangle-free graphs}
\author[Calum~Buchanan \and Puck~Rombach]{Calum~Buchanan \and Puck~Rombach}
\address{Dept.\:of Mathematics \& Statistics \\ University of Vermont \\ Burlington, VT, USA}
\email{$\{\mbox{calum.buchanan}, \mbox{puck.rombach}\}$@uvm.edu}
\date{April 14, 2025}


\begin{document}

\begin{abstract}
The saturation number $\operatorname{sat}(n, H)$ of a graph $H$ and positive integer $n$ is the minimum size of a graph of order $n$ which does not contain a subgraph isomorphic to $H$ but to which the addition of any edge creates such a subgraph.
Erd\H{o}s, Hajnal, and Moon first studied saturation numbers of complete graphs, and Cameron and Puleo introduced a general lower bound on $\operatorname{sat}(n,H)$.
In this paper, we present another lower bound on $\operatorname{sat}(n, H)$ with strengthenings for graphs $H$ in several classes, all of which include the class of triangle-free graphs.
Demonstrating its effectiveness, we determine the saturation numbers of diameter-$3$ trees up to an additive constant; these are double stars $S_{s,t}$ of order $s + t$ whose central vertices have degrees $s$ and $t$.
Faudree, Faudree, Gould, and Jacobson determined that $\operatorname{sat}(n, S_{t,t}) = (t-1)n/2 + O(1)$.
We prove that $\operatorname{sat}(n,S_{s,t}) = (st+s)n/(2t+4) + O(1)$ when $s < t$.
We also apply our lower bound to caterpillars and demonstrate an upper bound on the saturation numbers of certain diameter-$4$ caterpillars.
\end{abstract}

\maketitle

\section{Introduction}

Let $G$ and $H$ be finite, simple graphs.
If no subgraph of $G$ is isomorphic to $H$, we say that $G$ is {\em $H$-free}.
A number of foundational results in extremal graph theory concern global properties of $H$-free graphs.
Tur\'an's theorem, for instance, states that the complete $t$-partite graph whose partite sets are as balanced as possible uniquely contains the maximum number of edges out of all $K_{t+1}$-free graphs of a given order $n$.
Erd\H{o}s, Hajnal, and Moon~\cite{erdos1964problem} studied a complementary problem, proving that there is a unique graph of minimum size over all maximally $K_{t+1}$-free graphs of order $n$: 
the complete $t$-partite graph which is as unbalanced as possible (that is, with $t-1$ singleton partite sets and one of cardinality $n - t + 1$).
The study of maximally $K_{t+1}$-free graphs was initiated by Zykov, who called these graphs ``saturated"~\cite{zykov1949some}.

These works sparked interest in {\em $H$-saturated} graphs, those whose edge sets are maximal with respect to being $H$-free, for more general graphs $H$.
The minimum size of an $H$-saturated graph of order $n$ is called the {\em saturation number of $H$}, denoted $\sat(n,H)$.

In 1986, K\'aszonyi and Tuza introduced a general upper bound on $\sat(n, H)$~\cite{kaszonyi1986saturated}.
Notably, their bound implies that $\sat(n,H) = O(n)$ for any graph $H$ (with at least one edge).
In 2022, Cameron and Puleo proved a general lower bound via a weight function on the edge set of $H$~\cite{cameron2022lower}.
For each edge $uv$ in $H$, assuming $d(u) \leq d(v)$, let 
\[\wtcp(uv) = 2 |N(u) \cap N(v)| + |N(v) - N(u)| - 1.\]
They proved that there is a constant $c$ depending only on $H$ such that, for any integer $n$ at least the order of $H$, 
\begin{equation}\label{eq:CP}
\sat(n, H) \geq \left( \min_{uv \in E(H)} {\wtcp(uv)} \right) \frac{n}{2} - c.
\end{equation}
For a triangle-free graph $H$ or, more generally, a graph in which every edge $uv$ minimizing $\wtcp$ is contained in no triangles, \eqref{eq:CP} reduces to 
\begin{equation}~\label{eq:trianglefreeCP}
\sat(n, H) \geq \left(\min_{uv \in E(H)} {\max{\big\{ d(u), d(v) \big\}}} - 1 \right) \frac{n}{2} - c. 
\end{equation}
In what follows, we provide a different strengthening of~\eqref{eq:trianglefreeCP} for a general graph $H$ by considering not only the maximum degree of an endpoint of each edge in $H$, but also the maximum degree of a neighbor of one of its endpoints.
To do so, we define two weight functions on the edge set of $H$.

\begin{defn}\label{def:weights}
For each edge $uv$ in a graph $H$, let
\[
\wtzero(uv) = \max{\big\{ d(u), d(v) \big\}} - 1.
\]
If $uv$ is not isolated, that is, if $N(u) \cup N(v) \neq \{u,v\}$, let
\[
\wtone(uv) = 
\max{ \big\{ d(w) : w \in ( N(u) \cup N(v) ) - \{u,v\} \big\} }.
\]
Let $k_0$ and $k_1$ denote the minimum values of $\wtzero$ and $\wtone$, respectively, over $E(H)$.
Further, let
\[
k_0' = \min_{\substack{uv \in E(H) \\ \wtone(uv) = k_1}}{ \wtzero(uv) } \qquad \mbox{ and } \qquad k_1' = \min_{\substack{uv \in E(H) \\ \wtzero(uv) = k_0}}{ \wtone(uv) }.
\]
\end{defn}

Note that~\eqref{eq:trianglefreeCP} can be rewritten as $\sat(n,H) \geq k_0 n / 2 - c$.
We also note that, if $H$ has an isolated edge, then the bounds~\eqref{eq:CP} and~\eqref{eq:trianglefreeCP} are both trivial.
This, however, is the only such case, for $\sat(n,H) = O(1)$ if $H$ contains an isolated edge, and otherwise $\sat(n,H) = \Theta(n)$~{\cite{kaszonyi1986saturated}}.
The bounds we provide here concern saturation numbers which grow linearly with $n$.
Thus, we assume throughout that $H$ has no isolated edges, in addition to the obvious assumption that $H$ has at least one edge.

If an edge $xy$ is added to an $H$-saturated graph $G$, then this edge is contained in a copy of $H$ in $G + xy$.
Suppose $xy$ is the image of the edge $uv$ in $H$ in one such copy.
Naturally, the vertices $x$ and $y$ in $G$ must resemble $u$ and $v$ in a number of ways.
For instance, one of $x$ or $y$ is the image of the higher-degree endpoint of $uv$, so $\max{\{ d_{G + xy}(x), d_{G + xy}(y) \}} \geq \max{ \{ d_H(u), d_H(v) \}} \geq \wtzero(uv) + 1$.
Further, not only must we have $\max{\{ d_G(x), d_G(y) \}} \geq \wtzero(uv)$, but at least one of $x$ or $y$ must have a neighbor of degree at least $\wtone(uv)$.
It follows that almost all vertices in $G$ have degree at least $k_0$ and a neighbor of degree at least $k_1$ (see Propositions~\ref{prop:lowdegreeclique} and~\ref{prop:neighborclique}).
In Section~\ref{sec:general}, we use these observations, and similar ones involving $k_0'$ and $k_1'$, to obtain a first improvement on~\eqref{eq:trianglefreeCP} for a general graph $H$.

Noting that there are no restrictions on the degrees of neighbors in an $H$-saturated graph when $k_1' \leq k_0$, we focus on the case $k_1' > k_0$.
We also note that $k_0 \leq k_0'$ and $k_1 \leq k_1'$, and the former inequality is strict if and only if the latter is strict as well.
We summarize our general lower bound in the following theorem, which combines Lemmas~\ref{lem:generalpart1} and~\ref{lem:generalpart2} in Section~\ref{sec:general}.

\begin{thm}\label{thm:general}
Let $H$ be a graph with at least one edge and no isolated edges.
There is a constant $c$ depending only on $H$ such that, for any $n \geq |H|$,
\[
\sat(n,H) \geq \left( k_0 + \frac{k_1' - k_0}{k_1' + 1} \right) \frac{n}{2} - c.
\]
Further, if $k_1 > k_0$, then $\sat(n,H) \geq \big( k_0 + (k_1' - k_0) / k_1' \big) n / 2 - c$, and if $k_0 = k_1 < k_1' < k_0'$, then
\[
\sat(n,H) \geq
\begin{cases}
\displaystyle{ \left( k_0 + \frac{k_0' - k_0}{k_0' + 1} \right) \frac{n}{2} - c } : & \displaystyle{ k_0' \leq k_1' + \frac{k_0' - k_0}{k_0 + 1} } ; \\[10pt]
\displaystyle{ \left( k_0 + \frac{k_1' - k_0}{k_1'} \right) \frac{n}{2} - c } : & \text{otherwise.} 
\end{cases}
\]
\end{thm}

In Section~\ref{sec:trianglefree}, we {strengthen Theorem~{\ref{thm:general}}.
First, we assume that none of the edges in $H$ which minimize $\wt_0$ are contained in any triangles. In this case, for any pair of nonadjacent vertices $x$ and $y$ with degrees at most $k_0$ in an $H$-saturated graph $G$, at least one must have a neighbor $z$ with $|N_G(z) - \{x,y\}| \geq k_1' - 1$ (see Figure~{\ref{fig:trianglefree_common_nbr}}).
Second, we assume that at least one degree-$(k_0 + 1)$ endpoint, say $v$, of any given edge $uv$ in $H$ which minimizes $\wt_0$ has a neighbor of degree at least $k_1'$ and is not contained in any triangles.
In this case, since one of $x$ or $y$ is the image of such a vertex $v$ in a copy of $H$ in $G + xy$,}
there exists either $z \in N(y)$ with $|N(z) - (N(y) \cup x)| \geq k_1'$ or $z' \in N(x)$ with $|N(z') - (N(x) \cup y)| \geq k_1'$ (see Figure~{\ref{fig:triangle-free_high-degree}}).
The following theorem summarizes our two main strengthenings of Theorem~\ref{thm:general} for such graphs $H$, phrased in terms of triangle-free graphs for succinctness.

\begin{thm}\label{thm:trianglefree}
Let $H$ be a triangle-free graph with at least one edge and no isolated edges. If $k_1' \geq k_0 + \sqrt{2k_0 + 1}$, or if $k_1' \geq k_0 + 2$ and at least one degree-$(k_0 + 1)$ endpoint of every edge in $H$ minimizing $\wt_0$ has a neighbor of degree at least $k_1'$,
then there is a constant $c$ depending only on $H$ such that, for any $n \geq |H|$,
\[
\sat(n,H) \geq \left( k_0 + \frac{k_1' + 1 - k_0}{k_1' + 2} \right) \frac{n}{2} - c.
\]
If, in addition to either of the above conditions, $k_1 > k_0$, then
\[
\sat(n,H) \geq \left( k_0 + \frac{k_1' + 1 - k_0}{k_1' + 1} \right) \frac{n}{2} - c.
\]
\end{thm}

Theorem~\ref{thm:trianglefree} follows from Lemmas~\ref{lem:trianglefreepart1} and~\ref{lem:trianglefreepart2} in Section~\ref{sec:trianglefree}.
From their proofs, one can also obtain strengthenings of Theorem~\ref{thm:general} which do not require the conditions $k_1' \geq k_0 + \sqrt{2k_0 + 1}$ or $k_1' \geq k_0 + 2$ in Theorem~\ref{thm:trianglefree}.
We note one such strengthening (for the case $k_1 > k_0$) in Corollary~\ref{cor:trianglefreepart2}, which we apply in an accompanying discussion of saturation numbers for certain classes of trees.
In Theorem~~\ref{thm:doublestar}, we determine the saturation numbers of unbalanced double stars up to an additive constant; in Corollary~\ref{cor:doublestar}, we determine these numbers precisely for sufficiently large $n$ meeting a divisibility condition (see Sections~\ref{subsec:doublestar} and~\ref{subsec:trianglefree_highdegnbr}).
In Theorem~~\ref{thm:shorty}, we apply Corollary~\ref{cor:trianglefreepart2} and prove an upper bound on the saturation numbers of certain diameter-$4$ caterpillars (see Section~\ref{subsec:shorty}).
These classes were examined in~\cite{faudree2009saturation}, and the saturation numbers of balanced double stars were determined up to an additive constant.

While Theorem~\ref{thm:trianglefree} strengthens Cameron and Puleo's lower bound, this is not always the case for Theorem~\ref{thm:general}.
Indeed, Theorem~{\ref{thm:general}} strengthens~{\eqref{eq:CP}} if and only if there exists an edge minimizing $\wt_0$ which is not contained in any triangle (and $k_1' > k_0$). On the one hand, if $\wt_0(uv) = k_0$ and $uv$ is in no triangle, then $\wtcp(H) \leq \wtcp(uv) = k_0$; on the other, if every edge minimizing $\wt_0$ is in a triangle, then $\wtcp(H) = \min_{uv \in E(H)}{\{ \wt_0(uv) + |N(u) \cap N(v)|\}} > k_0$. For example,
suppose that $H$ consists of a triangle with any number $\ell$ of pendant edges attached to one of its vertices.
The minimum value of $\wtcp$ is $2$ in this case, which asymptotically determines the saturation number, while Theorem~\ref{thm:general} tells us that the average degree of an $H$-saturated graph cannot be much less than $1 + (\ell + 1) / (\ell + 2) < 2$.
The degrees of second neighbors of an edge, as well as the number of triangles containing it, may be useful in determining saturation numbers of graphs with larger diameters than this one.
In this work, we consider only first neighborhoods of edges.

{Before proceeding, we note that the proofs of Theorems~{\ref{thm:general}} and~{\ref{thm:trianglefree}} do not actually use the fact that an $H$-saturated graph $G$ is $H$-free. All of our lower bounds thus hold for the more general \emph{semisaturation number of $H$}, or the minimum number of edges in a (not necessarily $H$-free) graph of order $n$ to which the addition of any extra edge creates a new copy of $H$. Since an $H$-saturated graph is also $H$-semisaturated, our upper bounds on the saturation numbers of unbalanced double stars and caterpillars also hold for their semisaturation numbers. Notably, not only do we determine the saturation numbers of unbalanced double stars asymptotically (and, in some cases, exactly), but also their semisaturation numbers.}

\subsection{Definitions and notations}

For the purposes of this paper, a graph $G$ is a pair of sets $(V,E)$, where $V$, or $V(G)$, is a finite set of vertices, and $E$, or $E(G)$, is a set of unordered pairs of distinct vertices.
The {\em order} and {\em size} of $G$, the numbers of its vertices and edges, respectively, are denoted by $|G|$ and $\| G \|$.
The {\em neighborhood} $N_G(v)$ of a vertex $v$ in $G$ is $\{w \in V : vw \in E\}$, and its {\em degree} $d_G(v)$ is $|N_G(v)|$.
When the graph $G$ is clear from context, we use the notations $N(v)$ and $d(v)$.
The minimum degree of a vertex in $G$ is denoted by $\delta(G)$ and the average degree over $V$ by $d(G)$.
For a nonempty subset $S$ of $V$, $d(S)$ denotes the average degree of a vertex in $S$; that is, $d(S) = \frac{1}{|S|} \sum_{v \in S} d(v)$.
As a convention, we let $d(\varnothing) = 0$.
For disjoint subsets $S$ and $T$ of $V$, $e(S,T)$ denotes the number of edges in $G$ between $S$ and $T$.

\section{Lower bounds for general graphs}\label{sec:general}

We begin with a graph $G$ which is not complete and is (semi)saturated with respect to an arbitrary graph $H$.
If $x$ and $y$ are nonadjacent vertices in $G$, then $xy$ is an edge in a copy of $H$ contained in $G + xy$.
The degrees of $x$ and $y$ must therefore be sufficiently large that $xy$ constitutes such an edge.
More precisely, the degree of at least one of $x$ or $y$ must be at least $k_0$, the minimum value taken by $\wtzero$ over $E(H)$.
This implies the following

\begin{prop}\label{prop:lowdegreeclique}
Let $H$ be a graph with at least one edge. The vertices in any $H$-saturated graph of degree strictly less than $k_0$ form a clique.
\end{prop}

A simple minimum degree bound, along with Proposition~\ref{prop:lowdegreeclique} and a derivative, shows that the average degree of an $H$-saturated graph $G$ of order $n \geq |H|$ is at least $k_0 - (k_0 + 1)^2 / (4n)$.
But, we can say more about pairs of nonadjacent vertices $x,y$ in $G$.
As we assume that $H$ has no isolated edges, the neighbors of $x$ and $y$ must also have sufficient degree that $xy$ constitutes an edge in $H \subseteq G + xy$.
Letting $\wtone$ and $k_1$ be as in Definition~\ref{def:weights}, we see that at least one of $x$ or $y$ has a neighbor of degree at least $k_1$.
Moreover, if $xy$ is to be the image of an edge $uv$ in $H$ minimizing $\wtzero$, then there must be a vertex in $N(x) \cup N(y)$ of degree at least $k_1'$, the minimum value taken by $\wtone$ over all edges in $H$ which minimize $\wtzero$.
We summarize the implications of these statements in the following

\begin{prop}\label{prop:neighborclique}
Let $H$ be a graph with at least one edge and no isolated edges. The vertices in any $H$-saturated graph with no neighbor of degree at least $k_1$ form a clique, and so do the vertices with degree at most $k_0$ and no neighbor of degree at least $k_1'$.
\end{prop}

Note that the cliques in Propositions~\ref{prop:lowdegreeclique} and~\ref{prop:neighborclique} are of orders at most $k_0$, $k_1$ and $k_0 + 1$, respectively, and thus have little effect on the size of an $H$-saturated graph with large order.
In other words, an $H$-saturated graph cannot have many fewer edges than a graph with minimum degree $k_0$ in which every minimum-degree vertex has a neighbor of degree at least $k_1'$ and in which every vertex has a neighbor of degree at least $k_1$.
An example of such a graph when $k_0 = 3$, $k_1' = 5$, and $k_1 \leq 3$ is shown in Figure~\ref{subfiga:kdelta}.
Figure~\ref{subfigb:kdelta} depicts the case $k_1 \in \{ 4, 5 \}$.
We preface our lower bounds on $\sat(n, H)$ by proving that these examples have minimum average degree over all such graphs.

\begin{figure}
\centering
\begin{subfigure}{.45\textwidth}
\centering
\begin{tikzpicture}
[semithick, every node/.style={circle, draw=black!100, fill=black!100, inner sep=0pt, minimum size=2.5pt}]

\node (0a) at (-2, 0) {};
\node (1a) at (-1, 1) {};
\node (2a) at (-1, .5) {};
\node (3a) at (-1, 0) {};
\node (4a) at (-1, -.5) {};
\node (5a) at (-1, -1) {};

\node (0b) at (1, 0) {};
\node (1b) at (0, 1) {};
\node (2b) at (0, .5) {};
\node (3b) at (0, 0) {};
\node (4b) at (0, -.5) {};
\node (5b) at (0, -1) {};

\foreach \i in {1,2,3,4,5}
{
\draw (0a) edge (\i a);
\draw (0b) edge (\i b);
\draw (\i a) edge (\i b);
}

\draw (1a) edge (2b);
\draw (2a) edge (3b);
\draw (3a) edge (4b);
\draw (4a) edge (5b);
\draw (5a) edge (1b);

\end{tikzpicture}
\caption{A graph with minimum degree $3$ and average degree $10/3$ in which every degree-$3$ vertex has a degree-$5$ neighbor.}
\label{subfiga:kdelta}
\end{subfigure}
\hfill
\begin{subfigure}{.45\textwidth}
\centering
\begin{tikzpicture}
[semithick, every node/.style={circle, draw=black!100, fill=black!100, inner sep=0pt, minimum size=2.5pt}]

\node (h0) at (0,.625) {};
\node (h1) at (0,-.625) {};
\node (l00) at (1, 1) {};
\node (l01) at (1, .75) {};
\node (l02) at (1, .5) {};
\node (l03) at (1, .25) {};
\node (l10) at (1, -.25) {};
\node (l11) at (1, -.5) {};
\node (l12) at (1, -.75) {};
\node (l13) at (1, -1) {};

\node (h2) at (3,.625) {};
\node (h3) at (3,-.625) {};
\node (l20) at (2, 1) {};
\node (l21) at (2, .75) {};
\node (l22) at (2, .5) {};
\node (l23) at (2, .25) {};
\node (l30) at (2, -.25) {};
\node (l31) at (2, -.5) {};
\node (l32) at (2, -.75) {};
\node (l33) at (2, -1) {};

\foreach \i in {0,1,2,3}
{
\foreach \j in {0,1,2,3}
{
\draw (h\j) edge (l\j\i);
}
\draw (l0\i) edge (l2\i);
\draw (l1\i) edge (l3\i);
}

\draw (h0) edge (h1);
\draw (h2) edge (h3);

\draw (l00) edge (l21);
\draw (l01) edge (l22);
\draw (l02) edge (l23);
\draw (l03) edge (l20);

\draw (l10) edge (l31);
\draw (l11) edge (l32);
\draw (l12) edge (l33);
\draw (l13) edge (l30);

\end{tikzpicture}
\caption{A graph with minimum degree $3$ and average degree $17/5$ in which every vertex has a degree-$5$ neighbor.}
\label{subfigb:kdelta}
\end{subfigure}
\caption{
Graphs as described in Example~\ref{ex:kdelta} whose average degrees meet the lower bounds in Proposition~\ref{prop:warmup}.
}
\label{fig:kdelta}
\end{figure}
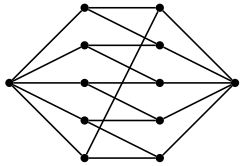
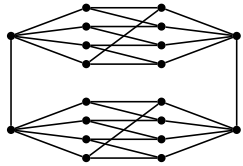

\begin{prop}\label{prop:warmup}
Let $\delta$ and $k$ be positive integers with $\delta < k$.
If $G$ is a graph with minimum degree $\delta$ in which every vertex of degree $\delta$ has a neighbor of degree at least $k$, then $d(G) \geq \delta + (k - \delta) / (k + 1)$.
If, in addition, every vertex in $G$ of degree at least $k$ has a neighbor of degree strictly larger than $\delta$, then $d(G) \geq \delta + ( k - \delta ) / k$.
\end{prop}

\begin{proof}
We partition the vertex set $V$ of $G$ as follows: let $S = \{ v \in V : d(v) = \delta \}$, $M = \{ v \in V : \delta < d(v) < k \}$, and $L = \{ v \in V : d(v) \geq k \}$.
By assumption, every vertex in $S$ has a neighbor in $L$, so $e(L, S) \geq |S| = |L \cup S| - |L|$.
Since $e(L, S) \leq \sum_{ v \in L } d(v) = d(L) |L|$, we have $|L \cup S| \leq (d(L) + 1) |L|$.
Let $\ell = d(L)$.
Combining the aforementioned inequalities yields 
\[
|L| \geq \frac{1}{\ell + 1} |L \cup S| \quad \mbox{ and } \quad |S| \leq \frac{\ell}{\ell + 1} |L \cup S|.
\]
Thus,
\begin{align*}
\sum_{v \in V} d(v) &\geq \ell |L| + \delta |S| + (\delta + 1) |M| \geq \frac{\ell (\delta + 1)}{\ell + 1} |L \cup S| + (\delta + 1) |M| \\ 
&\geq \left( \delta + \frac{\ell - \delta}{\ell + 1} \right) |G|.
\end{align*}
Since $\ell \geq k$ and
\begin{equation}\label{eq:star_dL}
\frac{\ell - \delta}{\ell + 1} = \frac{k - \delta}{k + 1} + \frac{(\delta + 1)(\ell - k)}{(\ell + 1)(k + 1)} 
\end{equation}
we have $d(G) = \sum d(v) / |G| \geq \delta + (k - \delta) / (k + 1)$, as desired.

For the second statement, if every vertex in $L$ has a neighbor in $V - S$, then $|S| \leq e(L, S) \leq \sum_{ v \in L} (d(v) - 1) = (\ell - 1) |L|$.
In this case, $|L| \geq |L \cup S| / \ell$ and $|S| \leq (\ell - 1) |L \cup S| / \ell$.
By a similar argument as before, we have
\[
\sum_{v \in V} d(v) \geq \left( \delta + \frac{ \ell - \delta }{ \ell } \right) |G|.
\]
Since
\begin{equation}\label{eq:doublestar_dL}
\frac{\ell - \delta}{\ell} = \frac{k - \delta}{k} + \frac{\delta (\ell - k)}{k\ell}
\end{equation}
we have $d(G) \geq \delta + (k - \delta)/k$, as desired.
\end{proof}

The lower bounds in Proposition~\ref{prop:warmup} are tight, as evidenced by the following constructions.
Such graphs will be relevant when we discuss saturation upper bounds in Section~\ref{sec:trianglefree}.

\begin{ex}\label{ex:kdelta}
It follows from the proof of Proposition~\ref{prop:warmup} that the graphs of minimum size in which every vertex has degree at least $\delta$ and a neighbor of degree at least $k > \delta$ are biregular, with all degrees either $\delta$ or $k$, and are such that every vertex has exactly one high-degree neighbor.
Note that such a graph contains $k - 1$ vertices of degree $\delta$ for every vertex of degree $k$, and thus its order must be divisible by $k$.
Further, the degree-$k$ vertices induce a matching, and thus its order must be divisible by $2k$. 
It is not difficult to construct such graphs.
For $n = 2 k \ell$, take $\ell$ copies of the double star $S_{k, k}$ (obtained from two copies of a star $K_{1,k-1}$ by attaching their central vertices with an edge), and put a $(\delta - 1)$-regular graph on the set of leaves (see Figure~\ref{subfigb:kdelta}).
If there are no restrictions on the degrees of neighbors of high-degree vertices, then every degree-$k$ vertex corresponds to $k$ degree-$\delta$ vertices in a graph of minimum size.
For $n = (k+1)\ell$, take $\ell$ copies of the star $K_{1,k}$ and, if at least one of $\delta - 1$, $k$, or $\ell$ is even, put a $(\delta-1)$-regular graph on the set of leaves (see Figure~\ref{subfiga:kdelta}).
Note that the bound in Proposition~\ref{prop:warmup} is strict whenever $n$ is not divisible by $2k$ in the first case or $k+1$ in the second case, or when $\delta - 1$, $k$ and $\ell$ are all odd in the second case.
\end{ex}

We now return to our discussion of saturation numbers.
As noted in the introduction, there are no restrictions on the degrees of neighbors in an $H$-saturated graph when $k_1' \leq k_0$, and we simply have $\sat(n,H) \geq k_0 n / 2 - (k_0 + 1)^2 / 8$.
Assuming $k_1' > k_0$, we now prove two lower bounds on $\sat(n,H)$, depending on $k_1$.

\begin{lem}\label{lem:generalpart1}
For any graph $H$ with at least one edge and no isolated edges, and for any $n \geq |H|$,
\[
\sat(n,H) \geq \left( k_0 + \frac{k_1' - k_0}{k_1' + 1} \right) \frac{n}{2} - c_1.
\]
If $k_1 > k_0$, then
\[
\sat(n,H) \geq \left( k_0 + \frac{k_1' - k_0}{k_1'} \right) \frac{n}{2} - c_2,
\]
where $c_1 = \frac{(k_0 + 1)(k_1' - k_0)}{2k_1' + 2} + \frac{(k_0 + 1)^2}{8}$ and $c_2 = \frac{(k_0 + 2)(k_1' - k_0)}{2k_1'} + \frac{(k_0 + 1)^2}{8}$.
\end{lem}

\begin{proof}
Let $G$ be an $H$-saturated graph of order $n$.
Partition the vertex set $V$ of $G$ as follows: let $S = \{ v \in V : d(v) \leq k_0 \}$, $M = \{ v \in V : k_0 < d(v) < k_1' \}$, and $L = \{ v \in V : d(v) \geq k_1' \}$.
By Propositions~\ref{prop:lowdegreeclique} and~\ref{prop:neighborclique}, aside from a clique $A$, every vertex in $S$ has degree $k_0$, and aside from a clique $B$, every vertex in $S$ has a neighbor in $L$.
Thus, $e(L,S) \geq |S - B| = |L \cup S| - |L| - |B|$, and clearly $e(L,S) \leq \sum_{v \in L} d(v) = |L| d(L)$.
Letting $\ell = d(L)$, it follows that $|L \cup S| - |B| \leq |L| (\ell + 1)$, so
\[
|L| \geq \frac{1}{\ell + 1} |L \cup S| - \frac{|B|}{\ell + 1} \quad \mbox{ and } \quad |S| \leq \frac{\ell}{\ell + 1} |L \cup S| + \frac{|B|}{\ell + 1}.
\]
Thus,
\[
\ell |L| + k_0 |S| \geq \frac{\ell + k_0 \ell}{\ell + 1} |L \cup S| - \frac{\ell - k_0}{\ell + 1} |B|
= \left( k_0 + \frac{\ell - k_0}{\ell + 1} \right) |L \cup S| - \frac{\ell - k_0}{\ell + 1} |B|.
\]
Using~\eqref{eq:star_dL} and noting that $|L \cup S| \geq |B|$, it follows that 
\[
\ell |L| + k_0 |S| \geq \left( k_0 + \frac{k_1' - k_0}{k_1' + 1} \right) |L \cup S| - \frac{k_1' - k_0}{k_1' + 1} |B|.
\]
Since $|B| \leq k_0 + 1$, we have
\begin{align*}
\sum_{v \in L \cup S} d(v) &= \ell |L| + k_0 |S - A| + \sum_{v \in A} d(v) \geq \ell |L| + k_0 |S| - |A| ( k_0 + 1 - |A|)\\
&\geq \left( k_0 + \frac{k_1' - k_0}{k_1' + 1} \right) |L \cup S| - \frac{(k_0 + 1)(k_1' - k_0)}{k_1' + 1} - \frac{(k_0 + 1)^2}{4}.
\end{align*}
Every vertex in $M$ has degree at least $k_0 + 1$ by definition, and $S$, $M$, and $L$ partition $V$, so the degree sum of $G$ is at least $\big(k_0 + (k_1' - k_0)/(k_1' + 1) \big) n - 2c_1$.

For the second statement, suppose that $k_1 > k_0$.
By Proposition~\ref{prop:neighborclique}, any vertices in $L$ with no neighbor of degree at least $k_1$ are adjacent.
Since $k_1 \leq k_1'$, there is at most one such vertex.
Thus, $e(L, S) \leq \sum_{v \in L} (d(v) - 1) + 1 = (\ell - 1) |L| + 1$.
Since $e(L,S) \geq |L \cup S| - |L| - |B|$, we have $|L \cup S| - |B| \leq \ell |L| + 1$.
If $L = \varnothing$, then $S = B$.
Otherwise,
\[
|L| \geq \frac{1}{\ell} |L \cup S| - \frac{1}{\ell} (|B| + 1) \quad \mbox{ and } \quad |S| \leq \frac{\ell - 1}{\ell} |L \cup S | + \frac{1}{\ell} (|B| + 1).
\]
Also, in this case, $|L \cup S| \geq |B| + 1$, so that using~\eqref{eq:doublestar_dL} we have
\[
\ell |L| + k_0 |S| \geq \left( k_0 + \frac{k_1' - k_0}{k_1'} \right) |L \cup S| - \frac{(k_0 + 2)(k_1' - k_0)}{k_1'}.
\]
Note that the above inequality still holds (and is strict) when $L = \varnothing$, in which case $\ell = 0$.
Thus, by the same reasoning with which we concluded the first paragraph of this proof, the degree sum of $G$ is at least $\big( k_0 + (k_1' - k_0) / k_1' \big) n - 2c_2$.
The handshake lemma completes the proof.
\end{proof}

We now consider graphs $H$ with $k_1 = k_0 < k_1' < k_0'$.
{Again, by Proposition~{\ref{prop:neighborclique}}, almost every vertex of degree at most $k_0$ in an $H$-saturated graph has a neighbor of degree at least $k_1'$.
In this case, however, almost every vertex of degree strictly less than $k_0'$ (in particular, almost every vertex of degree $k_1'$) has a neighbor of degree strictly larger than $k_1$, and $k_1 = k_0$.}
The constructions in Example~\ref{ex:kdelta} give two ideas for such a graph of minimum size, either we have only degree-$k_0$ and degree-$k_1'$ vertices with two degree-$k_1'$ vertices for every $2(k_1' - 1)$ degree-$k_0$ vertices (as in Figure~\ref{subfigb:kdelta}), or we have only degree-$k_0$ and degree-$k_0'$ vertices with one degree-$k_0'$ vertex for every $k_0'$ degree-$k_0$ vertices (as in Figure~\ref{subfiga:kdelta}).

\begin{ex}
Let us compare, as $k_0'$ varies, the average degree of a graph of the form given in Figure~\ref{subfiga:kdelta} with vertices of degree $k_0$ and $k_0'$ to the average degree of one as in Figure~\ref{subfigb:kdelta} with vertices of degree $k_0$ and $k_1'$.
Note that the former graph has average degree $k_0 + (k_0' - k_0) / (k_0' + 1)$ and the latter $k_0 + (k_1' - k_0) / k_1'$.
Suppose that $k_1' = 4$ and $k_0 = k_1 < k_1'$. 
If $k_0' = 6$, then $(k_0' - k_0)/(k_0' + 1) = 5/7 < (k_1' - k_0) / k_1' = 3/4$.
However, if $k_0' = 8$, then $(k_0' - k_0) / (k_0' + 1) = 7/9 > 3/4$.
If instead $k_0' = 7$, then the two quantities are equal.
In general, we have 
\begin{equation}\label{eq:k0prime}
\frac{k_0' - k_0}{k_0' + 1} \leq \frac{k_1' - k_0}{k_1'} \quad \mbox{if and only if} \quad k_0' - k_1' \leq \frac{k_0' - k_0}{k_0 + 1}.
\end{equation}
\end{ex}

We conclude this section, and the proof of Theorem~\ref{thm:general}, by determining that these constructions are optimal.
That is, when $k_0' > k_1'$, these graphs have minimum average degree over all graphs with minimum degree $k_0$ in which every degree-$k_0$ vertex has a neighbor of degree at least $k_1'$, and every vertex of degree strictly less than $k_0'$ has a neighbor of degree strictly greater than $k_0$.

\begin{lem}\label{lem:generalpart2}
For any graph $H$ with $k_0 = k_1 < k_1' < k_0'$, and for any $n \geq |H|$,
\[
\sat(n, H) \geq 
\begin{cases}
\big( k_0 + \frac{k_0' - k_0}{k_0' + 1} \big) \frac{n}{2} - c_1 : & k_0' \leq k_1' + \frac{k_0' - k_0}{k_0 + 1}; \\
\big( k_0 + \frac{k_1' - k_0}{k_1'} \big) \frac{n}{2} - c_2 : & k_0' \geq k_1' + \frac{k_0' - k_0}{k_0 + 1},
\end{cases}
\]
where $c_1 = \frac{(k_0 + 1)(k_0' - k_0)}{2k_0' + 2} + \frac{(k_0 + 1)^2}{8}$ and $c_2 = \frac{(k_0 + 2)(k_1' - k_0)}{2k_1'} + \frac{(k_0 + 1)^2}{8}$.
\end{lem}

\begin{proof}
Let $G$ be an $H$-saturated graph of order $n$.
Partition the vertex set $V$ of $G$ as follows: let $S = \{ v \in V : d(v) \leq k_0 \}$, $M = \{ v \in V : k_0 < d(v) < k_1' \}$, $L = \{ v \in V : k_1' \leq d(v) < k_0' \}$, and $\xl = \{ v \in V : d(v) \geq k_0'\}$.
Aside from a clique $B$, every vertex in $S$ has a neighbor of degree at least $k_1'$.
We partition $S - B$ into subsets $S_L$ and $S_{\xl}$ of vertices with a neighbor in $L$ or $\xl$, respectively (if a vertex has neighbors in both $L$ and $\xl$, assign it to either set arbitrarily).
We will show that $d(L \cup S_L)$ is not much less than $k_0 + (k_1' - k_0) / k_1'$ if $L$ is nonempty, and that $d( \xl \cup S_{\xl})$ is not much less than $k_0 + (k_0' - k_0) / (k_0' + 1)$ if $\xl$ is nonempty.

First, suppose $L \neq \varnothing$ and consider $L \cup S_L$.
At least one out of any pair of nonadjacent vertices in $L$ has a neighbor in $V - S$, since $d(v) < k_0'$ for all $v \in L$.
It follows that at most one vertex in $L$ has all of its neighbors in $S$, so that $|S_L| \leq e(L, S_L) \leq \sum_{v \in L} (d(v) - 1) + 1$.
That is, $|L \cup S_L| - |L| \leq |L| d(L) - |L| + 1$.
Letting $\ell = d(L)$, we have
\[
|L| \geq \frac{1}{\ell} |L \cup S_L| - \frac{1}{\ell} \quad \mbox{ and } \quad |S_L| \leq \frac{\ell - 1}{\ell} |L \cup S_L| + \frac{1}{\ell}.
\]
Thus,
\begin{align*}
\ell |L| + k_0 |S_L| &\geq \frac{\ell + k_0 (\ell - 1)}{\ell} |L \cup S_L| - \frac{\ell - k_0}{\ell} \\
&\geq \left( k_0 + \frac{k_1' - k_0}{k_1'} \right) |L \cup S_L| - \frac{k_1' - k_0}{k_1'}.
\end{align*}
Note that if $L = \varnothing$, the final inequality above still holds, and is strict.

Now consider $\xl \cup S_{\xl}$.
We have $|S_{\xl}| \leq e(\xl, S_{\xl}) \leq \sum_{v \in \xl} d(v)$.
Letting $x = d(\xl)$, we have
\[
|\xl| \geq \frac{1}{x + 1} |\xl \cup S_{\xl}| \quad \mbox{ and } \quad |S_{\xl}| \leq \frac{x}{x + 1} |\xl \cup S_{\xl}|.
\]
Thus,
\[
x |\xl| + k_0 |S_{\xl}| \geq \frac{x (k_0 + 1)}{x + 1} |\xl \cup S_{\xl}| \geq \left( k_0 + \frac{k_0' - k_0} {k_0' + 1} \right) |\xl \cup S_{\xl}|.
\]

We have
\[
\sum_{v \in V - M} d(v) = \big( x |\xl| + k_0 |S_{\xl}| \big) + \big( \ell |L| + k_0 |S_L| \big) + k_0 |B| - \sum_{s \in A} (k_0 - d(s)).
\]
If $k_0' - k_1' \geq (k_0' - k_0)/(k_0 + 1)$, then by~\eqref{eq:k0prime},
\[
\sum_{v \in V - M} d(v) \geq \frac{k_1' + k_0 (k_1' - 1)}{k_1'} | V - M | - \frac{k_1' - k_0}{k_1'} (|B| + 1) - |A| (k_0 + 1 - |A|).
\]
It follows that the degree sum of $G$ is at least $\big( k_0 + (k_1' - k_0)/k_1' \big) n - 2c_2$.
Otherwise, if $k_0' - k_1' \leq (k_0' - k_0)/(k_0 + 1)$, then
\[
\sum_{v \in V - M} d(v) \geq \frac{k_0' (k_0 + 1)} {k_0' + 1} | V - M | - \frac{ k_0' - k_0 } {k_0' + 1} |B| - |A| (k_0 + 1 - |A|).
\]
In this case, the degree sum of $G$ is at least $\big(k_0 + (k_0' - k_0) / (k_0 + 1) \big) n - 2c_1$.
The handshake lemma completes the proof.
\end{proof}

\section{Lower bounds for triangle-free graphs and saturation numbers of trees with small diameter}\label{sec:trianglefree}

Let $G$ be a graph which is not complete and is (semi)saturated with respect to a triangle-free graph $H$.
In Section~\ref{sec:general}, we concluded that, since every edge $uv$ in $H$ has a neighbor $w$ of degree at least $k_1$, at least one out of any pair of nonadjacent vertices in $G$ must have a neighbor of degree at least $k_1$.
Since $\{u,v,w\}$ is not a triangle, exactly one of the edges $uw$ or $vw$ is in $H$, and thus not only do we have $d(w) \geq k_1$, but $|N(w) - \{u,v\}| \geq k_1 - 1$.
Therefore, at least one out of any pair of nonadjacent vertices $x, y$ in $G$ must have a neighbor $z$ with $|N(z) - \{x,y\}| \geq k_1 - 1$.
Similarly, if $d(x)$ and $d(y)$ are both at most $k_0$, then not only must one of $x$ or $y$ have a neighbor of degree at least $k_1'$, but $G$ also has the following property:
\begin{equation*}
\tag{$P1$}\label{P1}
\parbox{\dimexpr\linewidth-4em}{{\emph{for any pair of nonadjacent vertices $x,y$ with degrees at most $k_0$, there exists $z \in N(x) \cup N(y)$ such that $|N(z) - \{x,y\}| \geq k_1' - 1$.}}}
\end{equation*}
{See Figure~{\ref{fig:trianglefree_common_nbr}}.}

\begin{figure}
    \begin{subfigure}{.45\linewidth}
        \centering
        \begin{tikzpicture}
        [semithick, every node/.style={circle, draw=black!100, fill=black!100, inner sep=0pt, minimum size=4pt}]
        
        \node (x) [label={[xshift=8pt, yshift=-7pt]$y$}] at (.5,-.25) {};
        \node (y) [label={[xshift=8pt, yshift=-7pt]$x$}] at (.5,-1.75) {};
        \node (z) [label={[xshift=1pt,yshift=1pt]$z$}] at (-1,-1) {};
    
        \draw (-2.75,-1) ellipse (.6 and 1.2);
        \node[draw=none,fill=none] (Nz) at (-2.75,-1) {$k_1' - 1$};
    
        \draw (z) -- (-2.57,.145);
        \draw (z) -- (-2.57,-2.145);
       
        \draw (x) -- (z);
        \draw (y) -- (z);
        \draw[dotted] (x) -- (y);
        
        \end{tikzpicture}
        \caption{A high-degree neighbor $z$ of nonadjacent low-degree $x,y$ needs at least $k_1' - 1$ neighbors outside of $\{x,y\}.$}
        \label{fig:trianglefree_common_nbr}
    \end{subfigure}
    \hfill
    \begin{subfigure}{.45\linewidth}
        \centering
        \begin{tikzpicture}
        [semithick, every node/.style={circle, draw=black!100, fill=black!100, inner sep=0pt, minimum size=4pt}]
        
        \node (y) [label={[xshift=8pt, yshift=-6pt]$y$}] at (0,0) {};
        \draw (y) circle (1.25);
        \node[draw=none,fill=none] at (360/8 : 1.75) {$N(y)$};
        \node (x) [label={[xshift=8pt, yshift=-7pt]$x$}] at (-.5,-1.75) {};
        \node (z) [label={[xshift=0pt,yshift=1pt]$z$}] at (-1,0) {};
    
        \node (extra) at (-.125,.75) {};
    
        \draw (-2.75,0) ellipse (.6 and 1.2);
        \node[draw=none,fill=none] (Nz) at (-2.75,0) {$k_1' - 1$};
    
        \draw (z) -- (-2.57,1.145);
        \draw (z) -- (-2.57,-1.145);
        
        \draw (y) -- (z);
        \draw[dotted] (x) -- (y);
        \draw (x) -- (z);
        \draw (z) -- (extra);
        \draw (y) -- (extra);
        
        \end{tikzpicture}
        \caption{A high-degree neighbor $z$ of low-degree vertex $y$ needs at least $k_1' - 1$ neighbors outside of $N(y)\cup \{x,y\}$.}
        \label{fig:triangle-free_high-degree}
    \end{subfigure}
    \caption{Illustrations of properties~\eqref{P1} and~\eqref{P2}.}
    \label{fig:P1andP2}
\end{figure}

Further, there is a subset $C$ of $k_0$ neighbors of either $x$ or $y$ such that $|N(z) - (C \cup \{x,y\})| \geq k_1 - 1$ for some $z \in N(x) \cup N(y)$.
A number of statements similar to those above can be made about the neighbors of nonadjacent vertices in a graph $G$ which is saturated with respect to a triangle-free graph $H$.

If every edge $uv$ in $H$ with $\wt_0(uv) = k_0$ has a degree-$(k_0 + 1)$ endpoint with a neighbor $w$ of degree $k_1'$, then we can make a stronger statement.
In this case, if $x$ and $y$ are nonadjacent vertices of degrees at most $k_0$ in an $H$-saturated graph $G$, then for at least one of them, having neighborhood $N$, there exists $z \in N$ with $|N(z) - (N \cup \{x,y\})| \geq k_1' - 1$.
This is because if, say, $y$ is to play the role of the degree-$(k_0 + 1)$ endpoint of an edge in $H$ minimizing $\wt_0$ with a degree-$k_1'$ neighbor, then every neighbor of $y$ is used, and none can make triangles with $z$, in this copy of $H$ (see Figure~{\ref{fig:triangle-free_high-degree}}).
Thus, in this case, $G$ has the following property:

\begin{equation*}
\tag{$P2$}\label{P2}
\parbox{\dimexpr\linewidth-4em}{{\em for any pair of nonadjacent vertices $x,y$ with degrees at most $k_0$, there exists either $z \in N(y)$ such that $|N(z) - (N(y) \cup x)| \geq k_1'$ or $z' \in N(x)$ such that $|N(z') - (N(x) \cup y)| \geq k_1'$.}}
\end{equation*}

We will use these properties in a similar manner as we used Proposition~\ref{prop:neighborclique} in the proofs of Lemmas~\ref{lem:generalpart1} and~\ref{lem:generalpart2}.
Throughout this section, given positive integers $k_0 < k_1'$ and a graph $G$ with property~\eqref{P1} or~\eqref{P2}, we refer to vertices of degree at least $k_1'$ as {\em high-degree vertices} and to those of degree at most $k_0$ as {\em low-degree}.

\begin{prop}\label{prop:trianglefree_neighborclique}
Let $k_0 < k_1'$, and let $G$ be a graph with property~\eqref{P2}.
The set of low-degree vertices $v$ in $G$ such that either
\begin{enumerate}[(i)]
\item $v$ has no high-degree neighbor, or
\item $v$ has a single high-degree neighbor $w$, $d(w) = k_1'$, and $N(v) \cap N(w) \neq \varnothing$
\end{enumerate}
is a clique.
\end{prop}

Let $G$ be a graph with minimum degree $k_0$. 
Suppose that either $k_1' \geq k_0 + 2$ and $G$ has property~\eqref{P2}, or $k_1' \geq k_0 + \sqrt{2k_0 + 1}$ and $G$ has property~\eqref{P1}.
If $G$ has minimum average degree over all such graphs, then
it follows from the proof of Lemma~\ref{lem:trianglefreepart1} in Section~\ref{subsec:trianglefree_lowdegvx_highdegnbr} that, aside from a small clique, the components of $G$ all resemble Figure~\ref{subfiga:kdelta} (that is, are as described in Example~\ref{ex:kdelta}), except that the high-degree vertices have degree $k_1' + 1$ and low-degree vertices have degree $k_0$.
If, in addition to the above conditions, every high-degree vertex in $G$ has a neighbor of degree strictly larger than $k_0$, it follows from the proof of Lemma~\ref{lem:trianglefreepart2} in Section~\ref{subsec:trianglefree_highdegnbr} that the components of $G$ all resemble Figure~\ref{subfigb:kdelta} (as in Example~\ref{ex:kdelta}), but with vertices of degree either $k_1' + 1$ or $k_0$.

\subsection{Low-degree vertices with high-degree neighbors}\label{subsec:trianglefree_lowdegvx_highdegnbr}

Let $H$ be a triangle-free graph without isolated edges.
Aside from a small clique, every low-degree vertex in an $H$-saturated graph has a high-degree neighbor.
Without considering the neighbors of high-degree vertices, Lemma~\ref{lem:generalpart1} implies that the average degree of an $H$-saturated graph cannot be much smaller than $k_0 + (k_1' - k_0)/(k_1' + 1)$, the average degree of a graph as described in Example~\ref{ex:kdelta}.
Notice, however, that the graph in Figure~\ref{subfiga:kdelta} does not have property~\eqref{P1} (and thus does not have~\eqref{P2}) when $k_0 = 3$ and $k_1' = 5$.
On the other hand, if we were to have $k_1' = 4$, then this graph would indeed have property~\eqref{P1} (although it would not necessarily be of minimum size).
In the following lemma, we show that, under some extra conditions on $H$, the average degree of an $H$-saturated graph cannot be much less than the average degree of a graph with minimum degree $k_0$ in which every low-degree vertex has a neighbor of degree $k_1' + 1$.

\begin{lem}\label{lem:trianglefreepart1}
Let $H$ be a triangle-free graph, and let $n \geq |H|$.
If $k_1' \geq k_0 + \sqrt{2k_0 + 9/4} - 1/2$, or if at least one degree-$(k_0 + 1)$ endpoint of every edge in $H$ minimizing $\wt_0$ has a neighbor of degree at least $k_1'$ and $k_1' \geq k_0 + 2$, then
\[
\sat(n,H) \geq \left( k_0 + \frac{k_1' + 1 - k_0}{k_1' + 2} \right) \frac{n}{2} - c,
\]
where $c = \frac{(k_0 + 1)(k_1' + 1 - k_0)}{2k_1' + 4} + \frac{(k_0 + 1)^2}{8}$.
\end{lem}

\begin{proof}
Let $G$ be an $H$-saturated graph on vertex set $V$ with order $n$.
Partition the high-degree vertices in $V$ into sets $L = \{ v \in V : d(v) = k_1' \}$ and $\xl = \{ v \in V : d(v) > k_1' \}$.
Further, let $S = \{ v \in V : d(v) \leq k_0 \}$ and $M = \{ v \in V : k_0 < d(v) < k_1' \}$.
Let $A$ denote the clique of vertices in $G$ with degree strictly less than $k_0$, and let $B$ denote the clique of vertices in {$S$ with no high-degree neighbor.}

We handle the degree sum over $\xl$ and the set $S_{\xl}$ of vertices in $S$ with a neighbor in $\xl$ in a nearly identical manner as we proved the first statement of Lemma~\ref{lem:generalpart1} or the second statement of Lemma~\ref{lem:generalpart2}.
We have $|S_{\xl}| \leq e(\xl, S_{\xl}) \leq d(\xl) |\xl|$ and $d(\xl) \geq k_1' + 1$ so that
\[
\sum_{v \in \xl \cup S_{\xl}} d(v) \geq \left( k_0 + \frac{k_1' + 1 - k_0}{k_1' + 2} \right) |\xl \cup S_{\xl}| - |A \cap S_{\xl} | (k_0 + 1 - |A|).
\]

We now restrict our attention to $L$ and the vertices in $S_L = S - (B \cup S_{\xl})$.
That is, $S_L$ is the set of vertices in $S - B$ whose only high-degree neighbor(s) lie in $L$.

\begin{case}
By the property~{\eqref{P1}}, if $x$ and $y$ are vertices in $S_L$ which share all of their neighbors in $L$, then $xy \in E(G)$.
It follows that, for any $z \in L$, the set of vertices $x$ in $N(z) \cap S_L$ whose only high-degree neighbor is $z$ form a clique (of order at most $k_0$).
Thus, at most $k_0 |L|$ vertices in $S_L$ have exactly one neighbor in $L$, so $2|S_L| - k_0|L| \leq e(L,S_L) \leq k_1'|L|$.
This gives $|L| \geq \frac{2}{k_1' + k_0 + 2} |L \cup S_L|$ and $|S_L| \leq \frac{k_1' + k_0}{k_1' + k_0 + 2} |L \cup S_L|$.
Therefore,
\[
k_1' |L| + k_0 |S_L| \geq \frac{2k_1' + k_0k_1' + k_0^2}{k_1' + k_0 + 2} |L \cup S_L|.
\] 
Note that
\[
\frac{2k_1' + k_0k_1' + k_0^2}{k_1' + k_0 + 2} \geq \frac{(k_0 + 1)(k_1' + 1)}{k_1' + 2}
\]
if and only if $k_1' \geq k_0 + \sqrt{2k_0 + 9/4} - 1/2$.
\end{case}

\begin{case}
Now, we suppose that at least one degree-$(k_0 + 1)$ endpoint of every edge minimizing $\wt_0$ in $H$ has a degree-$k_1'$ neighbor. In this case, $G$ has property~{\eqref{P2}}. We add to the clique $B$ all those vertices which do not meet condition (ii) of Proposition~{\ref{prop:trianglefree_neighborclique}} (those which share neighbors with their unique high-degree neighbor, which lies in $L$). In doing so, we remove all such vertices from $S_L$.

We claim that at most one neighbor in $S_L$ of any vertex in $L$ has exactly one neighbor in $L$.
Indeed, suppose for the sake of contradiction that, for some vertex $w \in L$, there exist vertices $u, v \in S_L$ with $N(u) \cap L = N(v) \cap L = \{w\}$.
In order for the pair $u,v$ not to contradict property~\eqref{P1}, we must have $uv \in E(G)$.
But then $u$ and $v$ both meet condition~(ii) of Proposition~\ref{prop:trianglefree_neighborclique}, so $\{u,v\} \subseteq B$. This is a contradiction, for then $u,v \notin S_L$ after all.

It follows from the previous claim and the pigeonhole principle that no more than $|L|$ vertices in $S_L$ have a single neighbor in $L$.
Thus,
\[
2|S_L| - |L| \leq e(L, S_L) \leq k_1' |L|.
\]
In other words, $2 |L \cup S_L| \leq (k_1' + 3) |L|$, so
\[
|L| \geq \frac{2}{k_1' + 3} |L \cup S_L| \quad \mbox{ and } \quad |S_L| \leq \frac{k_1' + 1}{k_1' + 3} |L \cup S_L|.
\]
Now,
\[
k_1' |L| + k_0 |S_L| \geq \frac{2 k_1' + k_0 (k_1' + 1)}{k_1' + 3} |L \cup S_L|.
\]
Note that
\[
\frac{ 2k_1' + k_0 (k_1' + 1) }{ k_1' + 3 } \geq \frac{ (k_0 + 1) (k_1' + 1) }{ k_1' + 2 }
\]
if and only if $k_1' \geq k_0 + 2$.
\end{case}

In both of the cases above, it follows that
\[
\sum_{v \in L \cup S_L} d(v) \geq \left( k_0 + \frac{k_1' + 1 - k_0}{k_1' + 2} \right) |L \cup S_L| - |A \cap S_L| (k_0 + 1 - |A|).
\]
Note that the degree sum over $S \cup L \cup \xl$ is the degree sum over $L \cup S_L$, $\xl \cup S_{\xl}$ and $B$, so
\[
\sum_{v \in S \cup L \cup \xl} d(v) \geq \left( k_0 + \frac{k_1' + 1 - k_0}{k_1' + 2} \right) | L \cup \xl \cup S | - \frac{k_1' + 1 - k_0}{k_1' + 2} |B| - |A| (k_0 + 1 - |A|).
\]
Noting that $d(v) \geq k_0 + 1$ for all $v \in M$ and that $S$, $M$, $L$, and $\xl$ partition $V$, we have
\[
\sum_{v \in V} d(v) \geq \left( k_0 + \frac{k_1' + 1 - k_0}{k_1' + 2} \right) n - \frac{(k_0 + 1)(k_1' + 1 - k_0)}{k_1' + 2} - \frac{(k_0 + 1)^2}{4},
\]
as desired.
\end{proof}

We remark that Lemma~{\ref{lem:trianglefreepart1}} applies to a larger class than that of triangle-free graphs. To use property~{\eqref{P1}} in Case~1, we require only that none of the edges in $H$ which minimize $\wt_0$ are contained in any triangles. If $k_1' \geq k_0 + \sqrt{2k_0 + 9/4} - 1/2$, then the lower bound in Lemma~{\ref{lem:trianglefreepart1}} holds for such a graph $H$. To use property~{\eqref{P2}} in Case~2, we also require that at least one degree-$(k_0 + 1)$ endpoint of every edge in $H$ minimizing $\wt_0$ has a neighbor of degree at least $k_1'$ and is not contained in any triangles.
If $k_1' \geq k_0 + 2$, then the lower bound holds for such a graph $H$.

\subsection{Double stars}\label{subsec:doublestar}

We detour here from our regularly scheduled programming to examine the saturation numbers of diameter-$3$ trees.
In particular, we prove that Lemma~\ref{lem:trianglefreepart1} is tight up to an additive constant for unbalanced double stars.
Given positive integers $s$ and $t$ with $s \leq t$, let $S_{s,t}$ denote the tree obtained from two stars $K_{1, s-1}$ and $K_{1,t-1}$ by connecting their central vertices with an edge (see Figure~\ref{subfig:doublestar}).
Note that, for $H = S_{s,t}$, we have $k_0 = s-1$, $k_1 = 1$, $k_0' = t - 1$, and $k_1' = t$.
{Further, every edge minimizing $\wt_0$ has a degree-$s$ endpoint with a neighbor of degree $t$.}

Faudree, Faudree, Gould, and Jacobson determined the saturation numbers of balanced double stars up to an additive constant and provided bounds for unbalanced double stars:
\begin{thm}[\cite{faudree2009saturation}]
For $n \geq s^3$, 
\begin{align*}
\frac{s-1}{2} n \leq \sat(n, S_{s, s}) &\leq \frac{s-1}{2} n + \frac{s^2 - 1}{2},
\quad \mbox{and} \\
\frac{s-1}{2} n \leq \sat(n, S_{s, t}) &\leq \frac{s}{2} n - \frac{(s - 1)^2 + 8}{8}.
\end{align*}
\end{thm}

In the following theorem, we use Lemma~\ref{lem:trianglefreepart1} and a construction resembling Figure~\ref{subfig:doublestarsat} to determine the saturation numbers of unbalanced double stars up to an additive constant.
Later, in Corollary~\ref{cor:doublestar} of Section~\ref{subsec:trianglefree_highdegnbr}, we show that the additive constant in the lower bound can be improved to match certain cases of the upper bound construction when $n$ is sufficiently large.

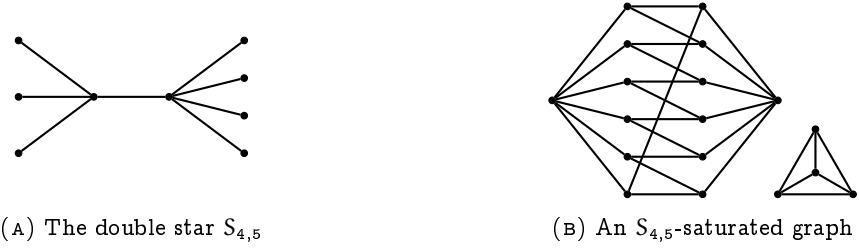
\begin{figure}
\centering
\begin{subfigure}{0.4\textwidth}
\centering
\begin{tikzpicture}
[every node/.style={circle, draw=black!100, fill=black!100, inner sep=0pt, minimum size=2.5pt}]

\node (0a) at (0, 0) {};
\node (1a) at (-1,-.75) {};
\node (2a) at (-1,0) {};
\node (3a) at (-1,.75) {};
\node (0b) at (1, 0) {};
\node (1b) at (2, -.75) {};
\node (2b) at (2, -.25) {};
\node (3b) at (2, .25) {};
\node (4b) at (2, .75) {};

\foreach \i in {1,2,3}
{
\draw[thick] (0a) edge (\i a);
\draw[thick] (0b) edge (\i b);
}

\draw[thick] (0a) edge (0b);
\draw[thick] (0b) edge (4b);

\draw[fill=none,draw=none] (0,-1.25) circle (.1);

\end{tikzpicture}
\caption{The double star $S_{4,5}$.}
\label{subfig:doublestar}
\end{subfigure}
\hfill
\begin{subfigure}{0.4\textwidth}
\centering
\begin{tikzpicture}
[every node/.style={circle, draw=black!100, fill=black!100, inner sep=0pt, minimum size=2.5pt}]

\node (0a) at (-2, 0) {};
\node (1a) at (-1, 1.25) {};
\node (2a) at (-1, .75) {};
\node (3a) at (-1, .25) {};
\node (4a) at (-1, -.25) {};
\node (5a) at (-1, -.75) {};
\node (6a) at (-1, -1.25) {};

\node (0b) at (1, 0) {};
\node (1b) at (0, 1.25) {};
\node (2b) at (0, .75) {};
\node (3b) at (0, .25) {};
\node (4b) at (0, -.25) {};
\node (5b) at (0, -.75) {};
\node (6b) at (0, -1.25) {};

\node (0c) at (1, -1.25) {};
\node (1c) at (2, -1.25) {};
\node (2c) at (1.5, -.384) {};
\node (3c) at (1.5, -.961) {};

\foreach \i in {1,2,3,4,5,6}
{
\draw[thick] (0a) edge (\i a);
\draw[thick] (0b) edge (\i b);
\draw[thick] (\i a) edge (\i b);
}

\draw[thick] (1a) edge (2b);
\draw[thick] (2a) edge (3b);
\draw[thick] (3a) edge (4b);
\draw[thick] (4a) edge (5b);
\draw[thick] (5a) edge (6b);
\draw[thick] (6a) edge (1b);

\foreach \i in {0,1,2}
{
\draw[thick] (\i c) edge (3c);
}
\draw[thick] (0c) -- (1c) -- (2c) -- (0c);

\end{tikzpicture}
\caption{An $S_{4,5}$-saturated graph.}
\label{subfig:doublestarsat}
\end{subfigure}
\caption{The double star $S_{4,5}$ on the left and an $S_{4,5}$-saturated graph on the right of order $n = 18$ and size $(12n - 6) / 7 = 30$.}
\label{fig:doublestar}
\end{figure}

\begin{thm}\label{thm:doublestar}
For a double star $S_{s,t}$ with $s < t$ and for $n \geq q(2t + 4) + s$ where $q = \max{\{1, \lfloor s / 2 \rfloor - 1\}}$,
\[
\left( \frac{s(t + 1)}{t + 2} \right) \frac{n}{2} - c_1 \leq \sat(n,S_{s,t}) \leq \left( \frac{s(t+1)}{t + 2} \right) \frac{n}{2} + c_2,
\]
where $c_1 = \frac{s(t - s + 2)}{2t + 4} + \frac{s^2}{8}$ and $c_2 = \frac{s(s-1)}{2t + 4} + \lceil \frac{s}{2} \rceil$.
\end{thm}

Our upper bound is based upon the observation that a graph $G_0$ obtained from two copies of $K_{1,t+1}$ by joining their sets of leaves with an $(s-2)$-regular bipartite graph, as in the larger component of Figure~{\ref{subfig:doublestarsat}}, is $S_{s,t}$-saturated and has average degree exactly $s(t+1)/(t+2)$.
Further, any graph consisting of disjoint copies of $G_0$ is $S_{s,t}$-saturated.
We are able to add a disjoint clique of cardinality $s$ to such a graph to obtain another $S_{s,t}$-saturated graph $G$ with 
\[\left(\frac{s(t+1)}{t+2}\right) \frac{n}{2} - \frac{s(t - s + 2)}{2t + 4}\]
edges, where $n = |G|$. In Corollary~{\ref{cor:doublestar}}, we show that this is the precise value of $\sat(n,S_{s,t})$ when $n$ is sufficiently large and equivalent to $s \pmod{2t+4}$.
When $n \not\equiv s \pmod{2t + 4}$, we add vertices to non-clique components of such a graph $G$ in a manner described below.

\begin{proof}[{Proof of Theorem~{\ref{thm:doublestar}}}]
The lower bound follows from Lemma~\ref{lem:trianglefreepart1}.
We provide a construction for the upper bound. Let $S_{s,t}$ be a double star with $s < t$. 
For $n \geq q(2t + 4) + s$ where $q = \max{\{1, \lfloor (s - 2) / 2 \rfloor\}}$, we construct an $n$-vertex graph $G$ with the following properties. 
\begin{enumerate}[(i)]
\item We have $V(G)=S \cup L$. For all $v\in S$, $d(v)=s-1$. For all $v\in L$, $d(v)\geq t+1$.
\item For all $v \in L$, $N(v) \subseteq S$, and every $w \in N(v)$ is contained in an independent set of cardinality $t+1$ in $N(v)$.
\item Aside from a clique $B$ of order $s$, every vertex in $S$ has a neighbor in $L$, and at most one vertex in $S$ has two or more neighbors in $L$.
\end{enumerate}

We claim that $G$ is $S_{s,t}$-saturated. Since there are no vertices of degree at least $t$ adjacent to any vertices of degree at least $s$, $G$ is $S_{s,t}$-free.
Let $x$ and $y$ be nonadjacent vertices in $G$.
If $x,y \in L$, then both have degree at least $t + 1$ by (i), and they have at most one common neighbor by (iii), so $x$ and $y$ are the internal vertices of a copy of $S_{s,t}$ in $G + xy$.
If $x \in S - B$, let $z \in N(x) \cap L$.
By (ii), there is an independent set $I_z$ of cardinality $t + 1$ in $N(z)$ which contains $x$. There are $t - 1$ vertices in $I_z - \{x,y\}$ and $s - 1$ neighbors of $x$ which are not in $I_z$.
Therefore, $x$ and $z$ are the internal vertices of a copy of $S_{s,t}$ in $G + xy$.
If $x \in B$, we may assume $y \in L$, in which case $B - x$ serves as a set of $s - 1$ leaves, and $y$ has a set of $t-1$ neighbors disjoint from $B$, resulting in a copy of $S_{s,t}$.

We construct $G$ as follows.
Let $L$ and $S$ partition the vertex set of $G$ with $|L| = 2 \lfloor (n - s) / (2t + 4) \rfloor$.
Let $r \equiv n - s \pmod{2t + 4}$, and let $R$ be a set of $r$ vertices in $S$.
Let $B$ be a clique of order $s$ in $S$.
Let every vertex in $L$ be adjacent to $t + 1$ distinct vertices in $S - (B \cup R)$ so that $V(G) - (B \cup R)$ induces a set of at least $2q$ copies of $K_{1,t + 1}$.
This partitions $S - (B \cup R)$ into classes.

If $r$ is even, make two of these stars into copies of $K_{1, t + 1 + r/2}$, and put an $(s - 2)$-regular bipartite graph on the two sets of $t + 1 + r/2$ vertices in $S$.
Since $|L|$ is even, we can pair up the remaining classes in $S - (B \cup R)$, and put an $(s-2)$-regular bipartite graph on each pair.

If $r$ is odd, let $v \in R$, and repeat the steps in the previous paragraph for $R - v$.
If $s$ is even, give $v$ a single neighbor in $L$, and if $s$ is odd, give $v$ two neighbors in $L$.
If $s > 3$, then take an adjacent pair in $S - B$, delete the edge between them, and give each an edge to $v$.
Repeat this, choosing a different pair of classes at each step for the adjacent pair to ensure condition (ii), until $v$ has degree $s-1$.
By our assumption on $n$, this is always possible, as there are at least $\lfloor s / 2 \rfloor - 1$ pairs of classes to choose from.

The resulting graph $G$ meets conditions (i)--(iii).
Further, for even $r$,
\begin{align*}
\| G \| &= \left( \frac{s(t+1)}{t+2} \right) \frac{n - r}{2} - \frac{s(t-s+2)}{2t+4} + \frac{sr}{2t+4} \\
&\leq \left( \frac{s(t+1)}{t+2} \right) \frac{n}{2} + \frac{s(s + t)}{2t+4},
\end{align*}
and for odd $r$,
\begin{align*}
\|G\| &= \left( \frac{s (t + 1)}{t + 2} \right) \frac{n-1}{2} - \frac{s (t - s + 2)}{2t + 4} + \frac{s(r-1)}{2t + 4} + \left\lceil \frac{s}{2} \right\rceil \\
&\leq \left( \frac{s (t + 1)}{t + 2} \right) \frac{n}{2} + \frac{s(s-1)}{2t+4} + \left\lceil \frac{s}{2} \right\rceil.
\end{align*}
This completes the proof.
\end{proof}

\subsection{High-degree neighbors}\label{subsec:trianglefree_highdegnbr}

We now consider triangle-free graphs $H$ with $k_1 > k_0$.
In a graph which is saturated with respect to such a graph $H$, almost every vertex has a neighbor of degree strictly larger than $k_0$.
We now strengthen Lemma~\ref{lem:trianglefreepart1} in a similar manner to the second statement in Lemma~\ref{lem:generalpart1}.
Two corollaries follow from the proof, one which removes the assumption $k_1' \geq k_0 + 2$ and the other regarding $S_{s,t}$-saturated graphs of sufficiently large order.

\begin{lem}\label{lem:trianglefreepart2}
Let $H$ be a triangle-free graph with $k_1 > k_0$, and let $n \geq |H|$.
If $k_1' \geq k_0 + \sqrt{2k_0 + 1}$, or if at least one degree-$(k_0 + 1)$ endpoint of every edge in $H$ minimizing $\wt_0$ has a neighbor of degree at least $k_1'$ and $k_1' \geq k_0 + 2$, then
\[
\sat(n,H) \geq \left( k_0 + \frac{k_1' + 1 - k_0}{k_1' + 1} \right) \frac{n}{2} - c,
\]
where $c = \frac{(k_0 + 2)(k_1' + 1 - k_0)}{2k_1' + 2} + \frac{(k_0 + 1)^2}{8}$.
\end{lem}

\begin{proof}
Let $G$ be an $H$-saturated graph with vertex set $V$ of order $n$, and let $S$, $M$, $L$, and $\xl$ partition the vertices $v \in V$ as in the proof of Lemma~\ref{lem:trianglefreepart1}: $S = \{v : d(v) \leq k_0\}$, $M = \{ v : k_0 < d(v) < k_1' \}$, $L = \{ v : d(v) = k_1' \}$, and $\xl = \{ v : d(v) > k_1' \}$.
Again, let $A$ denote the clique $\{ v : d(v) < k_0 \}$ and let $B$ denote the clique of vertices in $S$ with no high-degree neighbor.

At most one vertex in $L \cup \xl$ has all of its neighbors in $S$, for if there are two then they must be adjacent by Proposition~\ref{prop:neighborclique}.
Thus, letting $S_{\xl}$ denote the set of vertices in $S$ with a neighbor in $\xl$ and $x = d(\xl)$, we have $|S_{\xl}| \leq e(\xl, S_{\xl}) \leq (x - 1) |\xl| + 1$.
By the same arguments used to prove Lemmas~\ref{lem:generalpart1} and~\ref{lem:generalpart2}, we have
\[
x |\xl| + k_0 |S_{\xl}| \geq \left( k_0 + \frac{k_1' + 1 - k_0}{k_1' + 1} \right) |\xl \cup S_{\xl}| - \frac{k_1' + 1 - k_0}{k_1' + 1}.
\]

We now consider $L$ and the set $S_L$ of vertices in $S - B$ whose high-degree neighbors have degree exactly $k_1'$.
That is, $S_L = S - (B \cup S_{\xl})$.

\begin{case}
As in the proof of Lemma~{\ref{lem:trianglefreepart1}}, since $G$ has the property~{\eqref{P1}}, for any $z \in L$, the set of vertices in $N(z) \cap S_L$ whose only high-degree neighbor is $z$ form a clique (of order at most $k_0$).
It follows that $2 |S_L| - k_0 |L| \leq e(L, S_{L}) \leq (k_1' - 1)|L| + 1$.

This gives $2 |L \cup S_L| \leq (k_1' + k_0 + 1) |L| + 1$, and thus
\[
|L| \geq \frac{2 |L \cup S_L| - 1}{k_1' + k_0 + 1} \quad \mbox{ and } \quad |S| \leq \frac{(k_1' + k_0 - 1) |L \cup S_L| + 1}{k_1' + k_0 + 1}.
\]
Therefore,
\[
k_1' |L| + k_0 |S_L| \geq \frac{(k_0 + 2)k_1' + k_0(k_0 - 1)}{k_1' + k_0 + 1} |L \cup S_L| - \frac{k_1' - k_0}{k_1' + k_0 + 1}.
\]
Note that
\[
\frac{(k_0 + 2)k_1' + k_0(k_0 - 1)}{k_1' + k_0 + 1} \geq k_0 + \frac{k_1' + 1 - k_0}{k_1' + 1}
\]
if and only if $k_1' \geq k_0 + \sqrt{2k_0 + 1}$.
Thus,
\begin{equation*}\label{eq:LandSLboundpart1}
k_1' |L| + k_0 |S_L| \geq \left(k_0 + \frac{k_1' + 1 - k_0}{k_1' + 1} \right) |L \cup S_L| - \frac{k_1' - k_0}{k_1' + k_0 + 1}.
\end{equation*}
\end{case}

\begin{case}
Now, we suppose that that least one degree-$(k_0 + 1)$ endpoint of every edge minimizing $\wt_0$ in $H$ has a degree-$k_1'$ neighbor, in which case $G$ has property~{\eqref{P2}}.
As in the proof of Lemma~{\ref{lem:trianglefreepart1}}, we add to the clique $B$ all those vertices which do not meet condition (ii) of Proposition~{\ref{prop:trianglefree_neighborclique}} and, in so doing, remove these vertices from $S_L$.

By the same reasoning used to prove Lemma~\ref{lem:trianglefreepart1}, at most one neighbor in $S_L$ of any vertex in $L$ has a single high-degree neighbor.
Thus,
\[
2|S_L| - |L| \leq e(L, S_L) \leq (k_1' - 1) |L| + 1,
\]
which can be rewritten as $2 |L \cup S_L| \leq (k_1' + 2) |L| + 1$.
It follows that
\[
|L| \geq \frac{2}{k_1' + 2} |L \cup S_L| - \frac{1}{k_1' + 2} \quad \mbox{ and } \quad |S| \leq \frac{k_1'}{k_1' + 2} |L \cup S_L| + \frac{1}{k_1' + 2}.
\]
Noting that
\begin{equation}\label{eq:k0plus2part2}
\frac{2 k_1' + k_0 k_1'}{k_1' + 2} = k_0 + \frac{k_1' + 1 - k_0}{k_1' + 1} + \frac{k_1' (k_1' - k_0 - 1) - 2}{(k_1' + 1)(k_1' + 2)},
\end{equation}
whenever $k_1' \geq k_0 + 2$, in this case we have
\begin{align*}
k_1' |L| + k_0 |S_L| &\geq \frac{2 k_1' + k_0 k_1'}{k_1' + 2} |L \cup S_L| - \frac{k_1' - k_0}{k_1' + 2} \\
&\geq \left( k_0 + \frac{k_1' + 1 - k_0}{k_1' + 1} \right) |L \cup S_L| - \frac{k_1' - k_0}{k_1' + 2}.
\end{align*}
\end{case}

Now, in either Case~1 or Case~2, since $\frac{k_1' + 1 - k_0}{k_1' + 1} > \frac{k_1' - k_0}{k_1' + 2} \geq \frac{k_1' - k_0}{k_1' + k_0 + 1}$, and $L$ and $\xl$ cannot both have a vertex with all of its neighbors in $S$, we have
\[
x |\xl| + k_1' |L| + k_0 | S_L \cup S_{\xl} | \geq \left( k_0 + \frac{k_1' + 1 - k_0}{k_1' + 1} \right) |V - (M \cup B)| - \frac{k_1' + 1 - k_0}{k_1' + 1}.
\]
Finally, since $d(v) \geq k_0 + 1$ for all $v \in M$, $d(v) = k_0$ for all $v \in B$, and $|A|( k_0 + 1 - |A| ) \geq (k_0 + 1)^2 / 4$, we have
\[
\sum_{v \in V} d(v) \geq \left( k_0 + \frac{k_1' + 1 - k_0}{k_1' + 1} \right) n - \frac{(k_0 + 2)(k_1' + 1 - k_0)}{k_1' + 1} - \frac{(k_0 + 1)^2}{4},
\]
as desired.
\end{proof}

Just like Lemma~{\ref{lem:trianglefreepart1}}, the bound in Lemma~{\ref{lem:trianglefreepart2}} applies to a larger class of graphs $H$ than that of triangle-free graphs. To use property~{\eqref{P1}} in Case~1, we require that none of the edges in $H$ which minimize $\wt_0$ are contained in any triangles. To use property~{\eqref{P2}} in Case~2, we also require that a degree-$(k_0 + 1)$ endpoint of any such edge has a neighbor of degree at least $k_1'$ and is not contained in any triangles. Under the assumption $k_1' \geq k_0 + \sqrt{2k_0 + 1}$ in Case~1, and under the assumption $k_1' \geq k_0 + 2$ in Case~2, the lower bound on $\sat(n,H)$ in Lemma~{\ref{lem:trianglefreepart2}} holds, not only for triangle-free graphs, but any graph $H$ as described above.

To complete the proofs of Lemmas~\ref{lem:trianglefreepart1} and~\ref{lem:trianglefreepart2}, we needed $k_1'$ to be sufficiently large compared to $k_0$. However, strengthenings of Theorem~\ref{thm:general} can be obtained from these proofs for arbitrary values of $k_0$ and $k_1'$. We note one such strengthening below, which we apply to caterpillars in Section~\ref{subsec:shorty}. The interested reader can obtain strengthenings for arbitrary triangle-free graphs in a similar manner.

\begin{cor}\label{cor:trianglefreepart2}
Let $H$ be a triangle-free graph with $k_1 > k_0$, and let $n \geq |H|$.
If every edge minimizing $\wt_0$ in $H$ has a degree-$(k_0 + 1)$ endpoint with a neighbor of degree at least $k_1'$, then 
\[
\sat(n,H) \geq \left( k_0 + \frac{2}{k_0 + 3} \right) \frac{n}{2} - c,
\]
where $c = \frac{2k_0 + 3}{2k_0 + 6} + \frac{(k_0 + 1)^2}{8}$.
\end{cor}

\begin{proof}
Note that the assumption $k_1' \geq k_0 + 2$ was not used in the proof of Lemma~\ref{lem:trianglefreepart2} until~\eqref{eq:k0plus2part2}.
It follows from the proof that
\[
k_1' |L| + x|\xl| + k_0 |S_L \cup S_{\xl}| \geq \left( k_0 + \frac{2 (k_1' - k_0)}{k_1' + 2} \right) |V - (M \cup B)| - \frac{k_1' - k_0}{k_1' + 2}.
\]
Since $k_1' \geq k_1 \geq k_0 + 1$, and $d(v) = k_0$ for all $v \in B$, we have
\[
\sum_{v \in V} d(v) \geq \left( k_0 + \frac{2}{k_0 + 3} \right) |G| - \frac{2 |B| + 1}{k_0 + 3} - \frac{(k_0 + 1)^2}{4},
\]
and $2 |B| + 1 \leq 2k_0 + 3$, which completes the proof.
\end{proof}

In addition to providing an improved lower bound when $k_1 > k_0$, the techniques used in the proof of Lemma~\ref{lem:trianglefreepart2} can be used to show that a minimum $S_{s,t}$-saturated graph of sufficiently large order cannot have any vertices of degree strictly less than $s-1$.
In particular, when $n - s$ is divisible by $2t + 4$, the construction provided in the proof of Theorem~\ref{thm:doublestar} is optimal.
We note that, when $n - s$ is divisible by $t + 2$ and $t$ is odd, we can put an $(s-2)$-regular tripartite graph on the leaves of three stars $K_{1,t+1}$ to provide a simlar optimal construction when $n$ is sufficiently large.

\begin{cor}\label{cor:doublestar}
For any $s < t$, and for sufficiently large $n$,
\[
\sat(n, S_{s,t}) \geq \frac{s(t+1) n - s(t - s + 2)}{2t+4},
\]
and this is tight when $n \equiv s \pmod{2t + 4}$.
\end{cor}

\begin{proof}
Suppose that $G$ is an $S_{s,t}$-saturated graph of order $n$ and that the clique $A$ of vertices in $G$ with degree at most $s - 2$ is nonempty.
Let $v \in A$.
If $w$ is a nonneighbor of $v$, then $w$ must be the image of either the degree-$s$ or degree-$t$ vertex in the copy of $S_{s,t}$ in $G + vw$, and $v$ must be the image of a leaf.
Thus, $w$ has a neighbor of degree at least $s$.

Let $S$, $L$, and $\xl$ be as in the proofs of Lemmas~\ref{lem:trianglefreepart1} and~\ref{lem:trianglefreepart2}; that is, $S = \{ v : d(v) < s \}$, $L = \{ v : d(v) = t \}$, and $\xl = \{ v : d(v) > t \}$.
Further, let $S_L$, $S_{\xl}$, and $B$ partition $S$ in the same manner as Case~2 of either proof.
The vertex $v$ in $A$ has at most $s - 1 - |A|$ neighbors in $L \cup \xl$.
Let $C$ denote this set of high-degree neighbors of $v$.
We have $e(L, S_L) \leq (t - 1) |L| + |C \cap L|$ and $e(\xl, S_{\xl}) \leq (x-1) |\xl| + |C \cap \xl|$ where $x = d(\xl)$.
By similar reasoning to the proof of Lemma~\ref{lem:trianglefreepart2}, we have
\[
\sum_{v \in V(G)} d(v) \geq \left( s - 1 + \frac{ t - s + 2 }{t + 1} \right) n - \frac{|B \cup C| (t - s + 2)}{t + 1} - \frac{s^2}{4},
\]
and the right side of this inequality is strictly larger than
\[
\frac{s(t+1)n - s(t - s + 2)}{t + 2},
\]
when $n$ is sufficiently large.
Thus, in a minimum $S_{s,t}$-saturated graph $G$ of large order, the set $A$ is empty, and the desired lower bound on $\sat(n,S_{s,t})$ follows from the proof of Lemma~\ref{lem:trianglefreepart1}.
Tightness when $n \equiv s \pmod{2t + 4}$ follows from the upper bound construction in Theorem~\ref{thm:doublestar}.
\end{proof}

\subsection{Shorty the caterpillar and further remarks}\label{subsec:shorty}

Let us continue our discussion of saturation numbers of trees.
We would like to apply Lemma~\ref{lem:trianglefreepart2} to a tree with $k_1 > k_0$.
In order that this condition be met, the diameter must be at least $4$.
Consider the caterpillar $P_5^s$, obtained from a path of length $4$ by attaching $s$ pendant edges to each of the three internal vertices (see Figure~\ref{subfig:caterpillar}).
We name this caterpillar Shorty.
While the assumption that $k_1' \geq k_0 + 2$ in Lemma~\ref{lem:trianglefreepart1} is met by any unbalanced double star, the same assumption in Lemma~\ref{lem:trianglefreepart2} is not met by Shorty; every edge $uv$ has $\wtzero(uv) = s + 1$ and a degree-$(s+2)$ endpoint with a neighbor of degree $k_1'$, but $k_1' = s + 2$.
We thus apply Corollary~\ref{cor:trianglefreepart2} to obtain a lower bound on Shorty's saturation number.
The upper bound suggested by Lemma~\ref{lem:trianglefreepart2} does, however, hold for Shorty (see Figure~\ref{subfig:caterpillarsat}).

\begin{figure}
\centering
\begin{subfigure}{.4\textwidth}
\centering
\begin{tikzpicture}
[every node/.style={circle, draw=black!100, fill=black!100, inner sep=0pt, minimum size=2.5pt}]

\node (1) at (-2, 0) {};
\node (2) at (-1,0) {};
\node (3) at (0,0) {};
\node (4) at (1,0) {};
\node (5) at (2, 0) {};
\node (2a) at (-1, -.5) {};
\node (3a) at (0, -.5) {};
\node (4a) at (1, -.5) {};

\draw[thick] (1) -- (2) -- (3) -- (4) -- (5);
\draw[thick] (2) -- (2a);
\draw[thick] (3) -- (3a);
\draw[thick] (4a) -- (4);

\draw[fill=none,draw=none] (0,-1.25) circle (.1);

\end{tikzpicture}
\caption{The caterpillar $P_5^1$.}
\label{subfig:caterpillar}
\end{subfigure}
\hfill
\begin{subfigure}{.4\textwidth}
\centering
\begin{tikzpicture}
[every node/.style={circle, draw=black!100, fill=black!100, inner sep=0pt, minimum size=2.5pt}]

\node (h0) at (0,1.5) {};
\node (h1) at (0,0) {};
\node (l00) at (1, 2) {};
\node (l01) at (1, 1.5) {};
\node (l02) at (1, 1) {};
\node (l10) at (1, .5) {};
\node (l11) at (1, 0) {};
\node (l12) at (1, -.5) {};

\node (h2) at (3,1.5) {};
\node (h3) at (3,0) {};
\node (l20) at (2, 2) {};
\node (l21) at (2, 1.5) {};
\node (l22) at (2, 1) {};
\node (l30) at (2, .5) {};
\node (l31) at (2, 0) {};
\node (l32) at (2, -.5) {};

\node (t0) at (3.25, -.5) {};
\node (t1) at (4.25, -.5) {};
\node (t2) at (3.75, .366) {};

\foreach \i in {0,1,2}
{
\foreach \j in {0,1,2,3}
{
\draw[thick] (h\j) edge (l\j\i);
}
\draw[thick] (l0\i) edge (l2\i);
\draw[thick] (l1\i) edge (l3\i);
}

\draw[thick] (h0) edge (h1);
\draw[thick] (h2) edge (h3);

\draw[thick] (t0) -- (t1) -- (t2) -- (t0);

\end{tikzpicture}
\caption{A $P_5^1$-saturated graph.}
\label{subfig:caterpillarsat}
\end{subfigure}
\caption{The caterpillar $P_5^1$ on the left and a $P_5^1$-saturated graph on the right of order $n = 19$ and size $(5n - 3)/4 = 23$.}
\label{fig:shorty}
\end{figure}
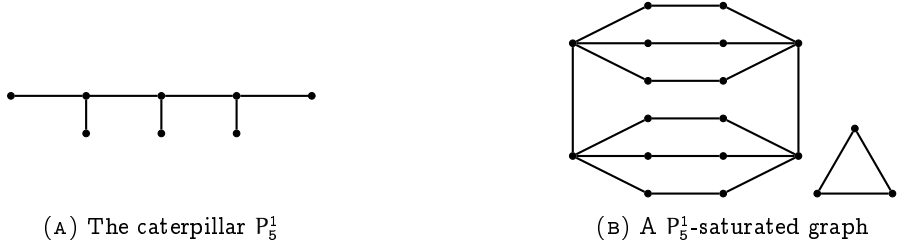

\begin{thm}\label{thm:shorty}
For any $s \geq 1$, and for any $n \geq q(2s+4) + s + 1$ where $q = \max{\{2, \lfloor (s-1)/2 \rfloor\}}$,
\[ 
\left( s + \frac{2}{ s + 3 } \right) \frac{n}{2} - c_1
\leq \sat(n,P_5^{s-1})
\leq \left( s + \frac{2}{ s + 2 } \right) \frac{n}{2} + c_2,
\]
where $c_1 = \frac{2s + 3}{2s + 6} + \frac{(s+1)^2}{8}$ and $c_2 = \frac{s(s + 1)}{s + 2} $.
\end{thm}

\begin{proof}
Let $H = P_5^{s - 1}$, so that $k_0 = s$ and $k_1 = k_1' = s+1$.
The lower bound follows from Corollary~\ref{cor:trianglefreepart2}.
We again provide a construction for the upper bound.
In particular, we construct an $n$-vertex graph $G$ with the following properties:
\begin{enumerate}[(i)]
\item We have $V(G) = S \cup L$.
For all $v \in S$, $d(v) = s$.
For all $v \in L$, $d(v) \geq s + 2$.
\item For all $v \in L$, $|N(v) \cap L| = 1$. Further, if $u,v \in L$ and $uv \in E(G)$, then $N(u) \cap N(v) = \varnothing$ and there are no edges between $N(u) \cap S$ and $N(v) \cap S$.
\item For all $v \in L$, every $w \in N(v) \cap S$ is contained in an independent set of cardinality $s + 1$ in $N(v) \cap S$.
\item Aside from a clique $B$ of order $s + 1$, every vertex in $S$ has a neighbor in $L$, and at most one vertex has two neighbors in $L$.
\end{enumerate}
Such a graph $G$ is $P_5^{s-1}$-free as there is no path of three consecutive vertices that have degree at least $s + 1$.
We now show that $G$ is $P_5^{s-1}$-saturated.  
Let $x$ and $y$ be nonadjacent vertices in $G$.
First, suppose $x,y \in L$.
By (i) and (ii), there exists $z \in N(x) \cap L$.
Further, $N(x) \cap S$ and $N(z) \cap S$ are disjoint sets of cardinality at least $s + 1$.
By (iv), $|N(y) \cap (N(x) \cup N(z))| \leq 1$,  
so there exists a set $I_y \subseteq N(y)$ of cardinality $s$,  which is disjoint from $N(x) \cup x$ and $N(z) \cup z$.
Let $I_x \subseteq N(x) \cap S$ be of cardinality $(s-1)$, and let $I_z \subseteq N(z) \cap S$ be of cardinality $s$.
We obtain a copy of $P_5^{s-1}$ in $G + xy$ with internal vertices $y$, $x$, and $z$, and leaves $I_y \cup I_x \cup I_z$.

Otherwise, at least one of $x$ or $y$ is in $S$. We assume, without loss of generality, that this vertex is $x$.
Suppose $x \in B$.
Let $I_x = B - x$.
By (ii) or (iv), depending on whether $y$ is in $L$ or $S - B$, respectively, there exists $z' \in N(y) \cap L$.
If $y \in L$, then $N(y) \cap N(z') = \varnothing$ by (ii), and there exist subsets $I_y \subset N(y)$ and $I_{z'} \subset N(z')$, of cardinalities $s-1$ and $s$, respectively, such that $I_x$, $I_y$, and $I_{z'}$ are pairwise disjoint.
On the other hand, if $y \in S$, then by (iii) there is an independent set $I \subseteq N(z')$ of cardinality $(s+1)$ that contains $y$.
Let $I_{z'} = I - y$.
Note that there exists a subset $I_y \subseteq N(y)$ of cardinality $(s-1)$ such that $I_{z'}$, $I_y$, and $I_x$ are pairwise disjoint.
We thus obtain a copy of $P_5^{s-1}$ in $G + xy$ with internal vertices $x$, $y$, and $z'$, and leaves $I_x \cup I_y \cup I_{z'}$.

Finally, suppose $x \in S - B$.
By (iv), there exists $z \in N(x) \cap L$, and by (ii), there is a single vertex $z'$ in $N(z) \cap L$.
By (iii), $x$ is contained in an independent set $I \subseteq N(z) \cap S$ of cardinality $s + 1$.
Let $I_z = I - \{x,y\}$.
If $z' \neq y$, let $I_x = N_{G + xy}(x) - z$.
Note $|I_x| = s$ by (i), and $I_x \cap I_z = \varnothing$.
In this case, we have $N(z') \cap N(z) = \varnothing$, $|N(z') \cap S| \geq s + 1$, and $I_x \cap N(z') = \varnothing$ by (ii).
Thus, $N(z') \cap S$ contains a subset $I_{z'}$ of cardinality $s$ which is disjoint from $I_z$ and $I_x$.
In this case, $x$, $z$, and $z'$ make up the internal vertices, and $I_x \cup I_z \cup I_{z'}$ the leaves, of a copy of $P_5^{s-1}$ in $G + xy$.
On the other hand, if $y = z'$, then $N(x) \cap N(y) = \{z\}$ by (ii).
Let $I_x = N(x) - z$.
Note that $|I_z| = s$ in this case.
By (i), $|I_x| = s - 1$, and there exists a subset $I_y \subset N(y) \cap S$ of cardinality $s$.
By (ii), $I_x \cap I_y$ and $I_y \cap I_z$ are both empty.
We have $I_x \cap I_z = \varnothing$ since $I_z \subset I$.
Thus, $z$, $x$, and $y$ make up the internal vertices, and $I_z \cup I_x \cup I_y$ the leaves, of a copy of $P_5^{s-1}$ in $G + xy$.
It follows that $G$ is $P_5^{s-1}$-saturated.

We construct $G$ as follows.
Let $S$ and $L$ partition the vertex set $V$ of $G$ with $|L| = 2 \lfloor (n - s - 1) / (2s + 4) \rfloor$.
Let $r \equiv n - s - 1 \pmod{2s + 4}$, and let $R$ be a set of $r$ vertices in $S$.
Let $B$ be a clique of order $s + 1$ in $S$.
Let every vertex in $L$ be adjacent to one other vertex in $L$ and to a distinct set of $s + 1$ vertices in $S - (B \cup R)$ so that $V - (B \cup R)$ is a set of at least $q$ double stars $S_{s + 2, s + 2}$.
If $r$ is even, make two of these double stars into copies of $S_{s + 2 + r/2, s + 2}$.
Put an $(s-1)$-regular bipartite graph on the two classes of size $s + 2 + r/2$.
If $|L|/2$ is even, put another $(s-1)$-regular bipartite graph on the two remaining classes in this pair.
Then, pair up the remaining double stars and similarly put $(s-1)$-regular bipartite graphs between classes which do not correspond to adjacent vertices in $L$ (as in Figure~\ref{subfig:caterpillarsat}).
If $|L| / 2$ is odd, then we make a triple of double stars, corresponding to three pairs of classes of vertices in $S$: $(A,B)$, $(C,D)$, and $(E,F)$.
Add three $(s-1)$-regular bipartite graphs with partite sets $(B,C), (D,E), (A,F)$.
Now pair up the remaining double stars as in the case where $|L|/2$ is even.

If $r$ is odd, let $v$ be a vertex in $R$.
Repeat the construction in the previous paragraph, replacing $R$ by $R - v$.
If $s$ is odd, give $v$ a single neighbor in $L$, and otherwise give $v$ two unmatched neighbors in $L$.
If $s > 2$, then take an adjacent pair of vertices in $S - B$, delete the edge between them, and give each an edge to $v$.
Repeat this, at each step choosing a different pair of classes for the adjacent pair in $S - B$ to ensure condition (iii), until $v$ has degree $s$.
By our assumption on $n$, this is always possible, as there are at least $\lfloor (s-1)/2 \rfloor$ pairs of classes to choose from.

The resulting graph $G$ meets conditions (i)--(iv).
Further, for even $r$,
\[
\| G \| = \left( s + \frac{2}{s + 2} \right) \frac{n}{2} - \frac{s+1}{s+2} + \frac{rs}{2s+4} \leq \left( s + \frac{2}{s + 2} \right) \frac{n}{2} + s - \frac{2s + 1}{s + 2}
\]
and for odd $r$,
\[
\| G \| = \left( s + \frac{2}{s + 2} \right) \frac{n-1}{2} - \frac{s+1}{s+2} + \frac{(r-1)s}{2s+4} + \left\lfloor \frac{s+2}{2} \right\rfloor.
\]
One can check that the right side of the above equation is at most $\big( s + 2/(s+2) \big) n / 2 + c_2$, which completes the proof.
\end{proof}

\begin{figure}
\centering
\begin{tikzpicture}
[every node/.style={circle, draw=black!100, fill=black!100, inner sep=0pt, minimum size=2.5pt}]

\node (h0) [label=left:{$z$ \ }] at (0,1.5) {};
\node (h1) [label=left:{$z'$}] at (0,0) {};
\node (h2) at (3,1.5) {};
\node (h3) at (3,0) {};
\node (l0) at (1.5,1.875) {};
\node (l1) at (1.5,-.375) {};
\node (l00) [label=above:{$x$}] at (1,1.125) {};
\node (l10) [label=below:{$y$}] at (1,.375) {};
\node (l20) at (2,1.125) {};
\node (l30) at (2, .375) {};

\foreach \j in {0,1,2,3}
{
\draw[thick] (h\j) edge (l\j0);
}

\draw[thick] (h0) edge (h1);
\draw[thick] (h2) edge (h3);
\draw[thick] (h0) edge (l0);
\draw[thick] (h2) edge (l0);
\draw[thick] (h1) edge (l1);
\draw[thick] (h3) edge (l1);
\draw[thick] (l00) edge (l20);
\draw[thick] (l10) edge (l30);
\end{tikzpicture}
\caption{A graph $G$ with property~\eqref{P2} for $k_0 = 2$ and $k_1' = 3$. Every vertex has a degree-$3$ neighbor, but $G$ is not $P_5^1$-saturated.}
\label{fig:shortylower}
\end{figure}
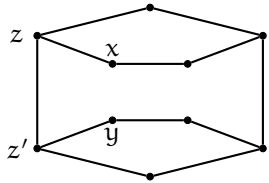

We note that the lower bound in Theorem~\ref{thm:shorty} applies to biregular caterpillars of arbitrary diameter.
That is, it applies to any caterpillar $P_\ell^{s-1}$ obtained from a path on $\ell \geq 5$ vertices by appending $s - 1$ leaves to each internal vertex.
{For $\ell \geq 7$, the degrees of second neighbors of the edges of $P_\ell^{s-1}$ will be relevant in determining their saturation numbers. We also note that any argument which holds for $P_\ell^0$ must use the $H$-free property of saturation, for it is known that the semisaturation number of $P_\ell$ is asymptotically less than its saturation number for $\ell \geq 6$ and $n \geq 3\ell - 3$~{\cite{burr2017study}}.}

We conclude with a remark on our lower bound for $P_5^{s-1}$ and a discussion of potential strengthenings.
A graph $G$ with property~\eqref{P2} ($k_0 = 2$, $k_1' = 3$) and $6n/5$ edges is depicted in Figure~\ref{fig:shortylower}.
There are at least two reasons why $G$ is not $P_5^1$-saturated, the former being that the pair of vertices $y,z$ does not meet the following property possessed by nonadjacent pairs in an $H$-saturated graph for triangle-free $H$: {there should be a subset $C$ of $N(y)$ or $N(z)$, $|C| = k_0$, and a vertex $w \in N(y) \cup N(z)$ such that $|N(w) - (C \cup \{y,z\})| \geq k_1 - 1$. It is possible that this property can be used to strengthen our lower bound on $\sat(n,P_\ell^s)$. A stronger lower bound may also} follow from a strengthening of Theorem~\ref{thm:trianglefree} for square-free graphs.
For example, consider the nonadjacent pair $x,y$ in $G$.
While this pair does not contradict property~\eqref{P2}, it is the lack of squares in $P_5^1$ which stops $xy$ from creating a copy in $G + xy$.
Indeed, $z \in N(x)$ and $|N(z) - (N(x) \cup y)| = 2$, but one of the vertices in $N(z) - (N(x) \cup y)$ is the high-degree neighbor $z'$ of $y$.
Since $y$ and $z'$ each have only one neighbor outside of $\{x,y,z,z'\}$, neither can play the role of a third high-degree vertex in $P_5^1$ with $x$ and $z$ as the other high-degree vertices.
By symmetry, the same is true if we had used $y$ and $z'$ instead of $x$ and $z$.

Lending credence to the idea that the upper bound in Theorem~\ref{thm:shorty} may be tight, we note that $P_5^0$ is simply a path of length $4$, and the saturated graphs of minimum size for paths are characterized in~\cite{kaszonyi1986saturated}.
In particular, a disjoint union of copies of $S_{3,3}$ is a minimum $P_5$-saturated graph, matching our construction.

\end{document}